\documentclass[11pt,twoside]{amsart}
\usepackage{amssymb}
\usepackage[utf8]{inputenc}
\usepackage{graphicx,color}
\usepackage{amsmath,amscd,amsthm,amssymb,amsxtra,latexsym,epsfig,epic,graphics}
\usepackage[all]{xy}
\nonstopmode

\numberwithin{equation}{section} 

\font\tengothic=eufm10 scaled\magstep 1 \font\sevengothic=eufm7
scaled\magstep 1
\newfam\gothicfam
      \textfont\gothicfam=\tengothic
      \scriptfont\gothicfam=\sevengothic

\newtheorem{Theorem}{Theorem}
\newtheorem{Lemma}[Theorem]{Lemma}
\newtheorem{Proposition}[Theorem]{Proposition}

\newtheorem{Conjecture}{Conjecture}

\theoremstyle{definition}
\newtheorem{Definition}[Theorem]{Definition}

\newtheorem{Remark}[Theorem]{Remark}
\newtheorem{Problem}[Theorem]{Problem}

\newtheorem{Example}[Theorem]{Example}

\newcommand{\KK}{{\mathbb K}}
\newcommand{\ZZ}{{\mathbb Z}}
\newcommand{\PP}{{\mathbb P}}

\newcommand{\QQ}{{\mathcal Q}}

\newcommand{\FF}{{\mathcal F}}
\newcommand{\EE}{{\mathcal E}}

\newcommand{\OO}{{\mathcal O}}

\newcommand{\rk}{\operatorname{rk}}

\makeindex

\begin{document}
\title[Uniform Steiner bundles]
{Uniform Steiner bundles}

\author[S. Marchesi, R.M.\ Mir\'o-Roig]{S. Marchesi, R.M.\
Mir\'o-Roig}

\address{Facultat de Matem\`atiques i Inform\`{a}tica,
Departament de Matem\`{a}tiques i Inform\`{a}tica, Gran Via de les Corts Catalanes
585, 08007 Barcelona, SPAIN } 
 \email{marchesi@ub.edu}
\email{miro@ub.edu}

\date{\today}


\begin{abstract} In this work we study $k$-type uniform Steiner bundles, being $k$ the lowest degree of the splitting. We prove sharp  upper and lower bounds for the rank in the case $k=1$ and moreover we give families of examples for every allowed possible rank and explain which relation exists between the families. After dealing with the case $k$ in general, we conjecture that every $k$-type uniform Steiner bundle is obtained through the proposed construction technique.  \end{abstract}

\maketitle

\section{Introduction}

Vector bundles on projective spaces are fundamental objects to study in order to understand the geometry and
topology of the varieties. In the study of vector bundles $\EE$ on $\PP^n$ an important technique is the restriction to a (general) hyperplane $H$ or to a (general) line $L$, since a lot of geometric information can be carried out from $\EE_{H}$ (resp. $\EE_{L}$) to $\EE$.

\vskip 2mm
Thanks to a well known result $\EE_L$ splits as a sum of line bundles $\EE_L\cong \oplus _{i=1}^r{\mathcal O}_{L}(a_i^L)$ and we call $(a_1^L\leq \dots \leq a_r^L)$ the splitting type of $\EE$ on $L$. Moreover, there exists a non-empty open subset $U\subset G(1,n)$ and a set of integers $(a_1\leq \dots \leq a_r)$ such that $\EE_L\cong  \oplus _{i=1}^r{\mathcal O}_{L}(a_i)$ for all $L\in U$. For every line $L \in U$, the given splitting is called the \emph{generic splitting type} of $\EE$ and the set $J(\EE)=G(1,n)\setminus U$ is called the set of \emph{jumping lines}. A vector bundle $\EE$ on $\PP^n$ is said to be {\rm uniform} if $J(\EE)=\emptyset$.\\ Uniform vector bundles have been studied for more than forty years and we will recall some among the most important results about them in the forthcoming section. So far only uniform bundles on $\PP^n$ of rank $r\le n+1$ have been classified and  a complete classification of uniform bundles of higher rank seems out of reach (see Section 2 for more details). In this paper we restrict our attention to Steiner bundles on $\PP^n$ and we address the problem of classifying uniform Steiner bundles with fixed splitting type. Moreover, we describe a technique that allows us to find families of uniform Steiner bundles whose defining matrix is explicit.

\vskip 2mm
The paper is organized as follows. In Section 2, we fix notation and we recall basic facts on homogeneous bundles, uniform bundles and Steiner bundles needed throughout the paper. Section 3 contains the main results of our work: we determine the possible ranks and first Chern classes of a 1-type  uniform bundle on $\PP^n$ (see Theorems 10 and 13) and how to connect two 1-type uniform Steiner bundles with fixed first Chern class. More precisely, given two 1-type uniform  bundles $\EE_1$ and $\EE_2$ on $\PP^n$ with $c_1(\EE_1)=c_1(\EE_2)$ and with associated matrices $A_1$ and $A_2$, respectively, we can get $A_2$ from $A_1$ by first adding a suitable number of linearly independent columns and then eliminating the right number of linear combination of columns (see Remark 15). In Section 4, we deal with $k$-type uniform Steiner bundles on $\PP^2$ and we get upper and lower bounds for their rank in terms of their first Chern class. We also explain a tricky construction of $k$-type uniform Steiner  bundles (see Example 16). In the last section we conjecture that  any $k$-type  uniform Steiner bundle can be obtained through our construction.

\vskip 2mm
\noindent \textbf{Acknowledgments:} The first author is partially supported by Fapesp Grant n. 2017/03487-9 and the program \emph{Research in Pairs} of the CRM in collaboration with the IMUB of the Universitat de Barcelona. The second  author is partially supported by  MTM2016--78623-P. The authors would like to thank  the Universidade Estadual de Campinas (IMECC-UNICAMP), the Universitat de Barcelona, the IMUB and the CRM for their kind hospitality.

\vskip 2mm
\noindent {\bf Notation.} In this paper $\mathbb{K}$ will be an algebraically closed field of characteristic 0, $V$ a $\mathbb{K}$-vector space of dimension $n+1$ and $\PP^n=\PP (V)$.


\section{Preliminaries}\label{sec-prel}

Let us start recalling  the definition of homogeneous bundles, uniform bundles as well as their relation and open questions related to them. To this end, we denote by $G(1,n)$ the Grassmannian of lines in $\PP^n$ and by $\ell $ the point in $G(1,n)$ which parameterizes the line $L\subset \PP^n$. According to Grothendieck's Theorem (see \cite{OSS}; Theorem 2.1) for any rank $r$ vector bundle $\EE$ on $\PP^n$ and for any $\ell \in G(1,n)$ there is an $r$-tuple
$$a_{\EE}(\ell)=(a_1(\ell),a_2(\ell),\dots ,a_r(\ell)), \quad a_1(\ell)\leq a_2(\ell)\leq \dots \leq a_r(\ell) $$
such that $\EE_{L}\cong \oplus _{i=1}^r{\mathcal O}_{L}(a_1(\ell))$. The $r$-tuple $a_{\EE(\ell )}$ is called {\em the splitting type of } $\EE$ {\em on} $\ell$. We have a map $$a_{\EE}:G(1,n)\longrightarrow \ZZ^r, \quad \ell \mapsto a_{\EE}(\ell)$$

\begin{Definition}\rm  \label{uniform} A vector bundle $\EE$ of rank $r$ on $\PP^n$ is {\em uniform} if $a_\EE$ is constant.
\end{Definition}

\begin{Example} \label{ExUni} \rm (1)  The sum of line bundles $\oplus {\mathcal O}_{\PP^n}(a_i)$ is uniform.

(2) The cotangent bundle $\Omega ^1_{\PP^n}$ is uniform since for any $\ell \in G(1,n)$ we have $a_{\Omega ^1_{\PP^n}}(\ell)=(-2,-1,\dots ,-1)$.
\end{Example}

\begin{Definition} \label{homogeneous} A vector bundle $\EE$ of rank $r$ on $\PP^n$ is {\em homogeneous} if for any projective transformation $t\in PGL(n+1,\KK)$ we have $t^*\EE\cong \EE$.
\end{Definition}

\begin{Example} \rm (1)  The sum of line bundles $\oplus \OO_{\PP^n}(a_i)$ is homogeneous.

(2) The tangent bundle $T_{\PP^n}$ is homogeneous since the differential of a projective transformation $t$ defines an isomorphism $T_{\PP^n}\cong t^*T_{\PP^n}$.
\end{Example}

Since, given two lines $L_1,L_2\subset \PP^n$, there is a projective transformation $t\in PGL(n+1,\KK)$ transforming $L_1$ onto $L_2$, we deduce that homogeneous vector bundles are always uniform. The notion of uniform vector bundle appeared first in the paper \cite{Sc} by Schwarzenberger, where he raised the question whether a uniform vector bundle $\EE$ on $\PP^n$ is homogeneous.\\ 
Let us summarize what we know.

\begin{Proposition} Let $\EE$ be a uniform rank $r$ vector bundle on $\PP^n$. If $r<n$ then $\EE$ splits into a sum of line bundles and, hence, it is homogeneous.
\end{Proposition}

\begin{proof} See \cite{OSS}; Theorem 3.2.3.
\end{proof}

This is no longer true for higher rank (see Example \ref{ExUni} (2)).
The characterization of rank $n$ (resp. $n+1$) uniform vector bundles on $\PP^n$ was given in  \cite{EHS}, \cite{S} and \cite{V} (resp. \cite{B}, \cite{E} and \cite{El}). The classification shows that every uniform bundle on $\PP^n$ of rank less or equal than $n+1$ is homogeneous. These results give rise to the following problem.

\begin{Problem} To determine the largest integer $r=r(n)$ such that any uniform vector bundle $\EE$ on $\PP^n$ of rank $r$ is homogeneous.
\end{Problem}

An example of non-homogeneous uniform bundle of rank $3n-1$ on $\PP^n$ was constructed by Hirschowitz (see \cite{OSS}; Theorem 3.3.2). Later this result was improved by Dr\'{e}zet who constructed for any integer $n\ge 2$ examples of rank  $2n$ non-homogeneous uniform bundles on $\PP^n$ in \cite{D}. So, we  have $r(n)\leq 2n-1$.

If we extend the study to fields of positive characteristic, examples of rank $n+1$ uniform non-homogeneus vector bundles on the $n$-dimensional projective space have been constructed by Xin in \cite{X}.

The study of uniform vector bundles has been also extended to other varieties: for example, in \cite{BN}, they are considered on quadric surfaces and, more recently, they are considered on Hirzebruch surfaces in \cite{BFF}, on Fano manifolds in \cite{MOS} and \cite{W}, and on multiprojective spaces in \cite{BM}.

Moreover, the interest in uniform vector bundles, and the construction of families of non-homogeneous examples, has been lately renewed for fields of positive characteristic (see \cite{X} as mentioned before), and because of their relation with the study of spaces of matrices of constant rank (see \cite{EM}).

In this paper, we will restrict our attention to Steiner bundles on $\PP^n$ and address the problem of characterizing  uniform Steiner bundles. 

So, let us finish this section recalling its definition and introducing the main object of study of our work.

\begin{Definition} \label{steiner}
 Let $\EE$ be a vector bundle  on $\PP^n$. We say that $\EE$ is a Steiner bundle if it is  defined by a short exact sequence
$$
0\rightarrow \EE \rightarrow \OO_{\PP^n}^x \rightarrow \OO_{\PP^n}(1)^c \rightarrow 0.
$$
\end{Definition}

Notice that the rank and the first Chern class of a Steiner bundle are fixed. Indeed, $c_1(\EE)=-c$ and $\rk(\EE)=x-c$. Steiner bundles could be seen as a generalization of  the cotangent bundle (for $\Omega ^1_{\PP^n}$ we have $c=1$ and $x=n+1$); the cotangent bundle is homogeneous, and therefore uniform, and we will study to what extend Steiner bundles are uniform and homogeneous.

Since $\EE$ is a vector bundle we have $x\ge n+c$. In addition, if we assume that $\EE$ does not contain $\OO_{\PP^n}$ as direct summand, we have $x\le (n+1)c$.

\begin{Definition}
Let $\EE$ be a uniform Steiner bundle on $\PP^n$. We call it a \textbf{$k$-type uniform} bundle if the splitting type, constant on every line, is of the form
$$
(-k^{a_k},-k+1^{a_{k-1}},\ldots,-1^{a_1},0^{a_0}), \:\:\:\:\mbox{with}\:\:\: a_k>0,
$$
where in the previous notation, $-i^{a_i}$ means that we have $a_i$ direct summands $\OO_{\PP^n}(-i)$ in the splitting.
\end{Definition}

Uniform Steiner bundles on $\PP^n$ of $k$-type exist. Indeed, the $k$-th symmetric power $S^k(\Omega ^1_{\PP^n}(1))$ of the cotangent bundle $\Omega ^1_{\PP^n}(1)$ is a $k$-type Steiner uniform bundle on $\PP^n$  since it fits in the exact sequence
$$
0\rightarrow S^k(\Omega ^1_{\PP^n}(1)) \rightarrow \OO_{\PP^n}^{n+k\choose k} \rightarrow \OO_{\PP^n}(1)^{n+k-1\choose k-1} \rightarrow 0
$$
and it has splitting type $(-k, -k+1^{{n-1\choose n-2}},-k+2^{{n \choose n-2}}, \ldots , -1^{{n-3+k\choose n-2}}, 0^{{n-2+k\choose n-2}})$ on every line. Notice that these $k$-type uniform Steiner bundles are homogeneous and our next goal will be to construct examples of non-homogeneous $k$-type uniform Steiner bundle on $\PP^n$.

For $k=1$, we will be able to determine the rank and Chern classes of any 1-type uniform Steiner bundle. Partial results will be achieved for $k>1$. In order to do so, we will construct families of uniform Steiner bundles, which are not homogeneous, giving the explicit defining matrix, and this represents in our opinion a novelty in the area. Indeed, for example, Elencwajg in \cite{E2} and Dr\'{e}zet in \cite{D} propose very elegant ways to construct non-homogeneous uniform bundles, respectively through monads and a geometric interpretation which will be later explained, but explicit examples do not appear.


\section{1-type uniform bundles}\label{sec-1type}
In this part, we will focus our attention on the first step, i.e. we will consider Steiner vector bundles which are uniform with splitting type $(-1^{a},0^b)$, first on $\PP^2$ and later on $\PP^n$, $n\ge 3$.

First of all, we construct a family of 1-type uniform bundles $\EE$ on $\PP^2$ which are different from the ones that we can get trivially by taking $\Omega_{\PP^2}(1)^a \oplus \OO_{\PP^2}^{b-a}$, i.e. $a$ copies of the cotangent bundle on the projective plane plus $b$ trivial summands.
After that, we determine the possible ranks for  1-type uniform Steiner bundles $\EE$ on $\PP^2$ with fixed first Chern class $c_1(\EE)$. We then obtain an upper and lower bound for the rank and we  also show  that the upper and lower bound are sharp. Finally we observe that, by the family constructed, we can achieve all the possible ranks in the admissible range.

Let us start constructing a family which  should be regarded as an example illustrating the idea behind the main
results of this paper.

\begin{Example}\label{ex-1type}
We will build a family of 1-type uniform bundles $\EE$ on $\PP^2$ which are  non-homogeneous when $c = -c_1(\EE)\geq 4$ (this last part will be proven in the following section). To this end, we consider the Steiner bundle $\EE$ on $\PP^2$ defined by the following exact sequence
$$
0\rightarrow \EE \rightarrow \OO_{\PP^2}^8 \stackrel{A}{\rightarrow} \OO_{\PP^2}(1)^3 \rightarrow 0
$$
where the morphism $A$ is represented by the matrix
$$
\left[
\begin{array}{ccccccccccc}
x & y & t & 0 & 0 & 0 & 0 & 0 \\
0 & 0 & 0 & x & y & t & 0 & 0 \\
t & 0 & 0 & 0 & t & 0 & x & y
\end{array}
\right]
$$
and  $x,y,t$ denote the homogeneous  coordinates of $\PP^2$. It is possible to prove by direct computation, or using Macaulay2, that $\EE$ is indeed a 1-type uniform Steiner bundle of rank 5 and $c_1(\EE)=-3$. Let us present the computation and strategy, which will be useful also for the case of higher values of $c=-c_1(\EE)$.

We consider the line $L\subset \PP^2$ defined by the equation $t= \alpha x + \beta y$. Through an action of the linear group which translates in a linear combination of the columns of $A$, and substituting the equation of the line in the matrix, it is possible to get
$$
A_L =
\left[
\begin{array}{ccccccccccc}
x & y & 0 & 0 & 0 & 0 & 0 & 0 \\
0 & 0 & 0 & x & y & 0 & 0 & 0 \\
0 & 0 & 0 & 0 & 0 & 0 & x & y
\end{array}
\right]
$$
which implies that $\EE_L \simeq \OO_{L}(-1)^3 \oplus \OO_L^2$, as wanted.\\
Let us now consider the line $L$ defined by $y = \alpha x$ and, using the same action, we apply the following transformations on the restriction of $A$, starting from
$$
\left[
\begin{array}{ccccccccccc}
x & \alpha x & t & 0 & 0 & 0 & 0 & 0 \\
0 & 0 & 0 & x & \alpha x & t & 0 & 0 \\
t & 0 & 0 & 0 & t & 0 & x & \alpha x
\end{array}
\right].
$$
We first vanish the $\alpha x$ entry in the fifth column in order to subsequently vanish the $t$ entry in the first one. Finally we arrive at the matrix
\begin{equation}\label{mat-L2}
\left[
\begin{array}{ccccccccccc}
x & 0 & t & 0 & 0 & 0 & 0 & 0 \\
0 & 0 & 0 & x & 0 & t & 0 & 0 \\
0 & 0 & 0 & 0 & t & 0 & x & 0
\end{array}
\right]
\end{equation}
which tells us that, also in this case, $\EE_L \simeq \OO_{L}(-1)^3 \oplus \OO_L^2$.\\
Finally, to study the line $x=0$, we apply the same technique as in the previous case and we can conclude that $\EE$ is actually a 1-type uniform Steiner bundle on $\PP^2$ with splitting type $(-1^3,0^2)$.

\vskip 4mm
We can generalize the previous example in order to have a $1$-type uniform Steiner bundle $\EE$ on $\PP^2$ of rank $c+2$ and first Chern class $c_1(\EE)=-c$, which is not homogeneous if $c\geq 4$, defined by the short exact sequence
$$
0\rightarrow \EE \rightarrow \OO_{\PP^2}^{2c+2} \stackrel{A}{\rightarrow} \OO_{\PP}(1)^c \rightarrow 0
$$
Let us construct the matrix $A$ iteratively from the previous one: let us add a row and two columns to get
$$
\left[
\begin{array}{ccccccccccc}
x & y & t & 0 & 0 & 0 & 0 & 0 & 0 & 0\\
0 & 0 & 0 & x & y & t & 0 & 0 & 0 & 0\\
t & 0 & 0 & 0 & t & 0 & x & y & 0 & 0 \\
0 & 0 & 0 & t & 0 & 0 & 0 & t & x & y
\end{array}
\right]
$$
Notice that the variable $x,y$ in the last row belong to the added columns, while we have non zero entries $t$ under two variables $x$ and $y$ which do not belong to the same line. Notice that we can describe the construction of the matrix as follows. Consider the matrix defining $c$ copies of $\Omega_{\PP^2}(1)$, i.e. a matrix with $c$ blocks of type $[x\:\:y\:\:t]$ put ``diagonally''. Erase then all columns which contain only the $t$ variable, except for the two columns where $t$ is in the first or second row. Then, starting from the third row on, we will add, again ``diagonally'', blocks of type $[t\:\:0]$ and $[0\:\:t]$ below the $[x\:\:y]$ blocks in the matrices (including the ones inside the two $[x\:\:y\:\:t]$ blocks that remained). \\
The described process defines the following $c\times (2c+2)$ matrix
$$
A = \left[
\begin{array}{cccccccccccccccccccc}
x & y & t & 0 & 0 & 0 & 0 & 0 & 0 & 0 & 0 & 0 & \cdots & & & \cdots& 0 & 0\\
0 & 0 & 0 & x & y & t & 0 & 0 & 0 & 0 & 0 & 0& \cdots & & & \cdots& 0 & 0\\
t & 0 & 0 & 0 & t & 0 & x & y & 0 & 0 & 0 & 0& \cdots & & & \cdots& 0 & 0\\
0 & 0 & 0 & t & 0 & 0 & 0 & t & x & y & 0 & 0& \cdots & & & \cdots& 0 & 0\\
0 & 0 & 0 & 0 & 0 & 0 & t & 0 & 0 & t & x & y& \cdots & &  & \cdots& 0 & 0\\
\vdots & & & & & & & & & & & \vdots && \ddots & & & &\vdots\\
0 & \cdots & & & & & & & & & \cdots& 0 &  t & 0 & 0 & t & x & y
\end{array}
\right]
$$
We claim that $A$ defines a $1$-type uniform Steiner  bundle in $\PP^2$ of rank $c+2$ and first Chern class $c_1(\EE)=-c$. Indeed, for every line $L$ defined by the equation $t = \alpha x + \beta y$ it is possible to apply linear combination of the columns of the restricted matrix in order to get
$$
A _L = \left[
\begin{array}{cccccccccccccccccccc}
x & y & 0 & 0 & 0 & 0 & 0 & 0 & 0 & 0 & 0 & 0 & \cdots & & & \cdots& 0 & 0\\
0 & 0 & 0 & x & y & 0 & 0 & 0 & 0 & 0 & 0 & 0& \cdots & & & \cdots& 0 & 0\\
0 & 0 & 0 & 0 & 0 & 0 & x & y & 0 & 0 & 0 & 0& \cdots & & & \cdots& 0 & 0\\
0 & 0 & 0 & 0 & 0 & 0 & 0 & 0 & x & y & 0 & 0& \cdots & & & \cdots& 0 & 0\\
0 & 0 & 0 & 0 & 0 & 0 & 0 & 0 & 0 & 0 & x & y& \cdots & &  & \cdots& 0 & 0\\
\vdots & & & & & & & & & & & \vdots && \ddots & & & &\vdots\\
0 & \cdots & & & & & & & & & \cdots& 0 &  0 & 0 & 0 & 0 & x & y
\end{array}
\right]
$$
which implies that $\EE_L \simeq \OO_{L}(-1)^c \oplus \OO_L^{2}$.

Let us consider now $L$ defined by $y = \alpha x$. Observing the position of the block $[0,t]$ in the last row, we can erase the entry $\alpha x$ which belongs in the same column of $t$ and we obtain a column $[0 \:\:\cdots \:\:0 \:\: t]^T$ with which we vanish the other $t$ entry in the last row. Therefore we can iterate this process in order to obtain a matrix of the same type of (\ref{mat-L2}), which also implies that $\EE_L \simeq \OO_{L}(-1)^c \oplus \OO_L^{2}$.

Finally, considering the line $x=0$, we apply a technique completely analogue to the previous computation and we arrive to the same conclusion, proving the 1-type uniformity of the bundle $\EE$.
\end{Example}

\begin{Theorem}\label{mainthm1}
Fix an integer $c\ge 2$. Let $\EE$ be  a 1-type uniform Steiner  vector bundle, with no trivial summands on $\PP^2$, of rank $r$, defined by the exact sequence
$$
0 \rightarrow \EE \rightarrow \OO_{\PP^2}^{r+c} \rightarrow \OO_{\PP^2}(1)^c \rightarrow 0.
$$
Then $c+2\leq r \leq 2c$. Moreover, these bounds are sharp and there exists a rank $r$  1-type uniform Steiner  vector bundle $\EE$ on $\PP^2$  with $c_1(\EE)=-c$ and $\rk(\EE)=r$ for any $r\in[c+2,2c]$.
\end{Theorem}
\begin{proof}
The upper  bound $r\le 2c$ immediately follows from the hypothesis that $\EE $ has no trivial summands. Let us now prove that $c+2\le r$. To this end, we consider the short exact sequence
$$
0 \rightarrow \EE(-1) \rightarrow \EE \rightarrow \EE_L \rightarrow 0
$$
and we recall that by hypothesis $\EE_L\simeq \OO_L(-1)^c \oplus \OO_L^{r-c}$, for every line $L \subset \PP^2$. We can observe directly that we can assume $r-c >0$ or else we would have that $\EE \simeq \OO_{\PP^2}(-1)^c$, as a consequence of \cite[Theorem 3.2.1]{OSS}. Moreover, by the uniformity hypothesis, we have that the map $f_L$ induced in the exact cohomology sequence
$$
H^1(\EE(-1)) \stackrel{f_L}{\longrightarrow} H^1(\EE) \longrightarrow H^1(\EE_L) =0
$$
is surjective for any line $L\subset \PP^2$.  Therefore we have a surjective morphism of vector bundles on the dual projective plane $(\PP^2)^*$
\begin{equation}\label{steinDual}
H^1(\EE(-1)) \otimes \OO_{(\PP^2)^*}(-1) \longrightarrow H^1(\EE) \otimes \OO_{(\PP^2)^*}\longrightarrow 0,
\end{equation}
whose kernel is a Steiner bundle, which we know must have rank bigger or equal than two, see \cite[Proposition 3.9]{DK}. The fact that $\EE$ does not have trivial summands is equivalent to asking that $H^0(\EE)=0$, from which  we can compute $h^1(\EE) = 2c -r$ and obviously $h^1(\EE(-1))=c$. Therefore we get $c-2c+r \geq 2$ or, equivalently,  $r\geq c+2$ as we wanted.

By Example \ref{ex-1type}, we know for any value of $c\ge 2$ the existence  of  rank $c+2$, 1-type uniform Steiner bundles $\EE_{c+2}$ on $\PP^2$ with first Chern class $-c$. So, the lower bound is achieved.
 We will construct the higher admissible rank cases as iterated extensions starting with  $\EE _{c+2}$ as described in the
  following diagram

$$
\xymatrix{
& 0 \ar[d] & 0 \ar[d] & 0 \ar[d]\\
0\ar[r] & \EE_{c+2}	\ar[r] \ar[d] & \OO_{\PP^2}^{2c+2} \ar[d] \ar[r]^{A} & \OO_{\PP^2}(1)^c \ar[d] \ar[r] & 0 \\
0\ar[r] & \EE_{c+2+i}\ar[r] \ar[d] & \OO_{\PP^2}^{2c+2+i} \ar[d] \ar[r]^{[A|C]} & \OO_{\PP^2}(1)^c \ar[d] \ar[r] & 0 \\
0 \ar[r] & \OO_{\PP^2}^{i} \ar[r] \ar[d] & \OO_{\PP^2}^{i} \ar[r] \ar[d] & 0\\
& 0 & 0
}
$$
where $[A|C]$ is constructed from $A$ adding $i$ linearly independent columns of linear forms. It is straightforward to notice that we can at most add $c-2$ linearly independent columns because $\dim H^1(\EE)=\dim H^0  \OO_{\PP^2}(1)^c -\dim H^0  \OO_{\PP^2}^{2c+2}= c-2$. So, when $i$ sweeps the interval $[0,c-2]$, we get all possible values of $r\in [c+2,2c]$.

Finally using the exact sequence
$$0 \longrightarrow \EE_{c+2} \longrightarrow \EE_{c+2+i} \longrightarrow  \OO_{\PP^2}^{i}  \longrightarrow 0$$ and the fact that $\EE_{c+2}$ is a 1-type uniform Steiner bundle we easily check hat $\EE _{c+2+i}$ is also a 1-type uniform Steiner bundle  which proves what we want.
\end{proof}
\begin{Remark}
The equality $r=2c$ in the previous result holds if and only if $\EE$ is the direct sum of $c$ copies of $\Omega_{\PP^2}(1)$.
\end{Remark}

In Theorem \ref{mainthm1} we have seen that, given a 1-uniform Steiner bundle $\EE$ on $\PP^2$ of rank $c+2 \leq r<2c$ and first Chern class $-c$, using iterate extensions of $\EE$ by $ \OO_{\PP^2}$ we construct a 1-uniform Steiner bundle $\EE_{max}$ on $\PP^2$ with first Chern class $-c$ and maximal rank, namely, $2c$. We would like to know if  given a 1-uniform Steiner bundle $\EE$ on $\PP^2$ of rank $c+2<r\leq 2c$ and first Chern class $-c$ we can also construct a 1-uniform Steiner bundle $\EE_{min}$ on $\PP^2$ with first Chern class $-c$ and minimal rank, namely, $c+2$.  This is possible and we will prove it more in general for the case of 1-type uniform Steiner bundles on $\PP^n$, when we will introduce a geometric interpretation of what is going on  inspired by the ideas developed by Dr\'{e}zet in \cite{D}.

Generalizing the family of bundles introduced for the case $n=2$, let us now concentrate on 1-type uniform vector bundles on $\PP^n$. First we construct a specific family, which we will shortly prove to be a family with the minimum allowed rank.

\begin{Example}\label{minRankPn}
Let us consider the Steiner bundle $\EE$ on $\PP^n$, $n\geq 2$, defined by the following short exact sequence
\begin{equation}\label{exPn}
0\longrightarrow \EE \longrightarrow \OO_{\PP^n}^{2c+2n-2} \stackrel{A_n}{\longrightarrow} \OO_{\PP^n}(1)^{c} \longrightarrow 0.
\end{equation}
We will now explain how to construct the defining matrix $A_n$, starting from what we already know in $\PP^2$. Indeed, we can write the case $n=2$, denoting by $x_0,x_1,x_2$ the variables, in the following form
$$
A_2 =
\left[
\begin{array}{cccccccccccccccccccccccccccc}
x_2 & 0 & x_0 & 0 & 0 & \cdots  &  & \cdots& 0 & \vline & x_1& 0 & & & & & 0\\
0 & x_2 & 0 & x_0 & 0 & &  & & \vdots & \vline & 0 & x_1& 0 & & & & \vdots\\
 0 & 0 & x_2 & 0 & x_0 & 0 & & &  & \vline &  0 & x_2 & x_1& 0 & & & \vdots\\
 \vdots & & \ddots& \ddots & \ddots & \ddots & \ddots& & \vdots& \vline & \vdots & & \ddots & \ddots & \ddots & &  \vdots\\
0& 0 & \cdots & 0 & x_2 & 0 & x_0 & 0 & 0 &  \vline & 0 & \cdots& 0 & x_2 & x_1& 0& 0\\
0& 0 & \cdots & 0 & 0 & x_2 & 0 & x_0 & 0 &  \vline & 0 & \cdots&0 & 0 & x_2 & x_1& 0\\
0& 0 & \cdots & 0 &  0 & 0 & x_2 & 0 & x_0 &  \vline & 0 & \cdots&0 & 0 & 0 & x_2 & x_1\\
\end{array}
\right]
$$
with $A_2$ a $c\times (2c+2)$-matrix.

Every time we want to add a further variable $x_i$, depending on the parity of $i$, we will add a diagonal of $0$'s and a diagonal of the variable $x_i$ both on the left side (if $i$ is even) or on the right side (if $i$ is odd) of the matrix $A_{i-1}$, and we will add such diagonal in order to only add two additional columns to the matrix $A_{i-1}$. Moreover we will insert an ``almost complete'' diagonal of variables $x_i$ on the complementary side, ``before'' the diagonal of the $(i-3)$-rd variable and multiplied by -1 if $i$ is odd and ``before'' the $(i-1)$-st variable when $i$ is even. Let us explain better such construction, specifying that by ``almost complete'' diagonal we mean without the first two elements when added on the right side and without the last two elements when added to the left side. Moreover, by ``before'' the diagonal of the previous matrix we mean closer to the central columns of the matrix.

An explicit example will help us understand the described process. Let us construct $A_3$ from $A_2$. We will add on  the right side the diagonals of $0$'s and of $x_3$ and we will add an almost complete diagonal of $-x_3$ on the left side.
$$
A_3 =
\left[
\begin{array}{cccccccccccccccccccccccccccc}
x_2 & 0 & x_0 & -x_3 & 0 & \cdots  &  & \cdots& 0 & \vline & x_1& 0 & x_3 & 0 & & & & & 0\\
0 & x_2 & 0 & x_0 & -x_3 & & & & \vdots & \vline & 0 & x_1& 0 & x_3 & 0 & & & & \vdots\\
 0 & 0 & x_2 & 0 & x_0 & -x_3 & & &  & \vline &  0 & x_2 & x_1& 0 & x_3 & 0 & & & \vdots\\
 \vdots & & \ddots& \ddots & \ddots & \ddots & \ddots& & \vdots& \vline & \vdots & & \ddots & \ddots & \ddots & \ddots & \ddots & &  \vdots\\
0& 0 & \cdots & 0 & x_2 & 0 & x_0 & -x_3 & 0 &  \vline & 0 & \cdots& 0 & x_2 & x_1& 0 & x_3& 0& 0\\
0& 0 & \cdots & 0 & 0 & x_2 & 0 & x_0 & 0 &  \vline & 0 & \cdots&0 & 0 & x_2 & x_1 & 0 & x_3 & 0\\
0& 0 & \cdots & 0 &  0 & 0 & x_2 & 0 & x_0 & \vline & 0 & \cdots&0 & 0 & 0 & x_2 & x_1& 0 & x_3\\
\end{array}
\right]
$$
Adding several variables, as described before, we obtain the following defining matrix (which we show for odd values of $n$, for even values simply consider $x_n=0$ and eliminate the last two columns on the right)
$$
A_n =
\left[
\begin{array}{cccccccccccccccccccccccccccc}
x_{n-1} & 0  & x_{n-3}& -x_n & & \cdots & x_2 & -x_5 & x_0 & -x_3 & 0 & \cdots   && \cdots& 0 & \vline\\
0& x_{n-1}  & 0 &  x_{n-3}& -x_n & \cdots  &-x_7 & x_2 & -x_5 & x_0 &- x_3 & & & & \vdots & \vline\\
 0 & 0 & x_{n-1} &0 &x_{n-3}  & \cdots & & -x_7 & x_2 & -x_5 & x_0 & -x_3 &  &&   \vdots & \vline\\
 \vdots & &  & & && & &  \ddots& \ddots & \ddots & \ddots & \ddots& & \vdots& \vline\\
0& 0 & \cdots &  & && &  & & -x_7 & x_2 &-x_5 & x_0 & -x_3 & 0 &  \vline\\
0& 0 & \cdots  & & & & & & & 0 & 0 & x_2 & 0 & x_0 & 0 &  \vline\\
0& 0 & \cdots  & & & & & & &0 &  0 & 0 & x_2 & 0 & x_0 & \vline\\
\end{array}
\right.
$$
$$
\:\:\:\:\:\:\:\:\:
\left.
\begin{array}{cccccccccccccccccccccccccccc}
\vline & x_1& 0 & x_3 & 0 & & & & & & &\cdots  & & & \cdots& 0& 0\\
\vline & 0 & x_1& 0 & x_3 & 0 & & & & & & \cdots  & & & \cdots& 0& 0\\
\vline &  0 & x_2 & x_1& x_4 & x_3 & x_6 & & & & & \cdots& & & \cdots& 0& 0\\
\vline & \vdots & & \ddots & \ddots & \ddots & \ddots & \ddots & & & & &  & & & & \vdots\\
\vline & 0 & \cdots& 0 & x_2 & x_1& x_4 & x_3& x_6& x_5 &  & \cdots & 0 & x_n & 0 & 0 & 0\\
\vline & 0 & \cdots&0 & 0 & x_2 & x_1 & x_4 & x_3 & x_6 & x_5 & \cdots &x_{n-2} & 0 & x_n & 0 & 0\\
\vline & 0 & \cdots&0 & & 0 & x_2 & x_1 & x_4 & x_3 & x_6 & \cdots & x_{n-1}&x_{n-2} & 0 & x_n & 0\\
\vline & 0 & \cdots&0 & & 0 & 0 & x_2 & x_1& x_4 & x_3 & \cdots & x_{n-4} & x_{n-1}&x_{n-2} & 0 & x_n \\
\end{array}
\right]
$$
which is of dimension $c \times (2c+2n-2)$ and it is the defining matrix of the short exact sequence in (\ref{exPn}).
We claim that $\EE$ is 1-type uniform vector bundle on $\PP^n$ which is not homogeneous for $c\geq 4$ (again not being homogeneous will be proven in the subsequent section). Indeed, let us first notice that, proving that the splitting is constant to each line is equivalent to prove that we have the same splitting type every time we fix any two of the variables and express the others as all possible linear combinations of the fixed ones. We will do the computations for an example, and, due to the specific construction of our matrices, the general case follows by iterating such process.

Let us consider the defining matrix
$$
\left[
\begin{array}{cccccccccccccccccccccccccccccccccccccccccccc}
x_4 & 0 & x_2 & -x_5 & x_0 & -x_3 & 0 & 0 & x_1 & 0 & x_3 & 0 & x_5 & 0 & 0 & 0\\
0 & x_4 & 0 & x_2 & -x_5 & x_0 & -x_3 & 0 & 0 & x_1 & 0 & x_3 & 0 & x_5 & 0 & 0\\
0 & 0 & x_4 & 0 & x_2 & 0 & x_0 & 0 & 0 & x_2 & x_1 & x_4 & x_3 & 0 & x_5 & 0\\
0 & 0 & 0 & x_4 & 0 & x_2 & 0 & x_0 & 0 & 0 & x_2 & x_1 & x_4 & x_3 & 0 & x_5
\end{array}
\right]
$$
Notice that the we can consider the variables as ``paired'' because of the symmetries in the matrix: $x_0$ with $x_1$, $x_2$ with $x_3$ and $x_4$ with $x_5$. This tells us that we do not have to consider all possible pairs, but it is sufficient to take $(x_0,x_i)$, for $i=1,\ldots,5$,  $(x_2,x_i)$, for $i=3,\ldots,5$, and $(x_4,x_5)$.
We will only explicit the case $(x_0,x_1)$; all the other cases follow from analogue computations.

First, we must express the other variables as linear combinations of $x_0$ and $x_1$, indeed, let us denote $x_i = a_i x_0 + b_i x_1$, for $i=2,\ldots,5$. Therefore, the previous matrix, restricted through the linear combinations and already applying operations on its columns, can be written as
$$
\left[
\begin{array}{cccccccccccccccccccccccccccccccccccccccccccc}
a_4 x_0  & 0 & a_2 x_0 & -a_5 x_0 & x_0 & -a_3 x_0 & 0 & 0 & \vline\\
0 & a_4 x_0 + b_4 x_1 & 0 & a_2 x_0 + b_2 x_1 & -a_5 x_0 - b_5 x_1 & x_0 & -a_3 x_0 - b_3 x_1 & 0 & \vline\\
0 & 0 & a_4 x_0 + b_4 x_1 & 0 & a_2 x_0 + b_2 x_1 & 0 & x_0 & 0 & \vline\\
0 & 0 & 0 & b_4 x_1 & 0 & b_2 x_1 & 0 & x_0 & \vline\\
\end{array}
\right.
$$
$$
\left.
\begin{array}{cccccccccccccccccccccccccccccccccccccccccccccc}
\vline & x_1 & 0 & a_3 x_0 & 0 & a_5 x_0 & 0 & 0 & 0\\
\vline & 0 & x_1 & 0 & a_3 x_0 + b_3 x_1 & 0 & a_5 x_0 + b_5 x_1 & 0 & 0\\
\vline & 0 & a_2 x_0 + b_2 x_1 & x_1 & a_4 x_0 + b_4 x_1 & a_3 x_0 + b_3 x_1 & 0 & a_5 x_0 + b_5 x_1 & 0\\
\vline & 0 & 0 & b_2 x_1 & x_1 & b_4 x_1 & b_3 x_1 & 0 & b_5 x_1
\end{array}
\right]
$$
If $a_4$ and $b_5$ are different from zero, then using the first (respectively last) column we can vanish all the other multiples of $x_0$ (of $x_1$) in the first row (last row); we get
$$
\left[
\begin{array}{cccccccccccccccccccccccccccccccccccccccccccc}
 x_0  & 0 & 0 & 0 & 0 & 0 & 0 & 0 & \vline\\
0 & a_4 x_0 + b_4 x_1 & 0 & a_2 x_0 + b_2 x_1 & -a_5 x_0 - b_5 x_1 & x_0 & -a_3 x_0 - b_3 x_1 & 0 & \vline\\
0 & 0 & a_4 x_0 + b_4 x_1 & 0 & a_2 x_0 + b_2 x_1 & 0 & x_0 & 0 & \vline\\
0 & 0 & 0 & 0 & 0 & 0 & 0 & x_0 & \vline\\
\end{array}
\right.
$$
$$
\left.
\begin{array}{cccccccccccccccccccccccccccccccccccccccccccccc}
\vline & x_1 & 0 & 0 & 0 & 0 & 0 & 0 & 0\\
\vline & 0 & x_1 & 0 & a_3 x_0 + b_3 x_1 & 0 & a_5 x_0 + b_5 x_1 & 0 & 0\\
\vline & 0 & a_2 x_0 + b_2 x_1 & x_1 & a_4 x_0 + b_4 x_1 & a_3 x_0 + b_3 x_1 & 0 & a_5 x_0 + b_5 x_1 & 0\\
\vline & 0 & 0 & 0 & 0 & 0 & 0 & 0 & x_1
\end{array}
\right]
$$
Now, using the $x_0$ in the 6th column and the $x_1$ in the 11th, we can erase such variables in the corresponding rows, obtaining
$$
\left[
\begin{array}{cccccccccccccccccccccccccccccccccccccccccccc}
 x_0  & 0 & 0 & 0 & 0 & 0 & 0 & 0 & x_1 & 0 & 0 & 0 & 0 & 0 & 0 & 0\\
0 & b_4 x_1 & 0 &  b_2 x_1 &   -b_5 x_1 & x_0 &  -b_3 x_1 & 0 &  0 & x_1 & 0 &  b_3 x_1 & 0 &  b_5 x_1 & 0 & 0\\
0 & 0 & a_4 x_0  & 0 & a_2 x_0  & 0 & x_0 & 0 & 0 & a_2 x_0  & x_1 & a_4 x_0 & a_3 x_0  & 0 & a_5 x_0  & 0\\
0 & 0 & 0 & 0 & 0 & 0 & 0 & x_0 & 0 & 0 & 0 & 0 & 0 & 0 & 0 & x_1
\end{array}
\right]
$$
Following an analogous reasoning that regards whether the other coefficients vanish or not, and studying all possible cases, it is possible to conclude that the splitting in this case is of type $(-1^4,0^{12})$.

\end{Example}

Thanks to the family presented in the previous example, we are now ready to prove the following result.
\begin{Theorem}\label{mainthm2}
Fix an integer $c\ge 2$. Let $\EE$ be  a 1-type uniform Steiner  vector bundle, with no trivial summands, on $\PP^n$  of rank $r$, defined by the exact sequence
$$
0 \rightarrow \EE \rightarrow \OO_{\PP^n}^{r+c} \rightarrow \OO_{\PP^n}(1)^c \rightarrow 0.
$$
Then $c+2(n-1)\leq r \leq cn$. Moreover, these bounds are sharp and there exists a rank $r$  1-type uniform Steiner vector bundle $\EE$ on $\PP^n$  with $c_1(\EE)=-c$ and $\rk(\EE)=r$ for any $r\in[c+2(n-1),cn]$.
\end{Theorem}
\begin{proof}
As before, the upper  bound $r\le cn$ immediately follows from the hypothesis that $\EE $ has no trivial summands and the equality $r=cn$ holds if and only if $\EE $ is $c$ copies of $\Omega ^1_{\PP^n}(1)$. \\
Let us now prove that $c+2(n-1)\leq r$. To do so, we consider the exact sequence obtained twisting by $\EE$ the Koszul exact sequence associated to $\OO_L$, in which we underline the first kernel from the left,
$$
\xymatrix{
0 \ar[r] & \EE(-n+1) \ar[r] & \EE(-n+2)^{n-1} \ar[r] & \cdots\\
\cdots \ar[r] & \EE(-2)^{\binom{n-1}{2}} \ar[r] & \EE(-1)^{n-1}\ar[dr] \ar[rr] & & \EE \ar[r] & \EE_L \ar[r] & 0\\
 & & & \mathcal{K} \ar[ur] \ar[dr] \\
 & & 0\ar[ur] && 0
}
$$
Also recall that by hypothesis $\EE_L\simeq \OO_L(-1)^c \oplus \OO_L^{r-c}$, for every line $L \subset \PP^n$ and notice directly that we can assume $r-c >0$ or else we would have that $\EE \simeq \OO_{\PP^n}(-1)^c$, as a consequence of \cite[Theorem 3.2.1]{OSS}. Being $\EE$ a Steiner bundle, we have that $H^1(\mathcal{K})\simeq H^1(\EE(-1)^{n-1})$; indeed, from the short exact sequence which defines $\EE$ we obtain the vanishing of the cohomology modules $H^2_*(\EE)=\cdots=H^{n-1}_*(\EE)=0$, with $H^i_*(\EE)= \bigoplus_{k \in \ZZ}H^i(\EE(k))$. The required isomorphism follows from the the definition of $\mathcal{K}$ given in the previous diagram. Moreover, by the uniformity hypothesis, we have that the map $f_L$ induced in the exact cohomology sequence
$$
H^1(\mathcal{K}) \stackrel{f_L}{\longrightarrow} H^1(\EE) \longrightarrow H^1(\EE_L) =0
$$
is surjective for any line $L\subset \PP^n$. Therefore we have a surjective morphism of vector bundles on the Grassmannian of lines $G :=G(1,n)$
\begin{equation}\label{steinDualPn}
H^1(\EE(-1)^{n-1}) \otimes \OO_{G} \longrightarrow H^1(\EE) \otimes \QQ \longrightarrow 0,
\end{equation}
where $\QQ$ denotes the universal quotient bundle of rank $n-1$ in the Grassmannian $G$, and whose kernel is a Steiner bundle, which we know must have rank bigger or equal than $2n-2=\dim G$ (see \cite{AM}). The fact that $\EE$ does not have trivial summands is equivalent to asking for $H^0(\EE)=0$, from which  we can compute $h^1(\EE) = cn -r$ and obviously $h^1(\EE(-1))=c$. Therefore we get $c(n-1) - cn +r \geq 2n-2$ or, equivalently,  $r\geq c+2(n-1)$ as we wanted.

By Example \ref{minRankPn}, we know for any value of $c\ge 2$ the existence  of  rank $c+2(n-1)$, 1-type uniform Steiner bundle $\EE_{c+2(n-1)}$ on $\PP^n$ with first Chern class $-c$. As before, the lower bound is achieved and we will construct the higher admissible rank cases as iterated extensions by trivial bundles, starting with  $\EE _{c+2(n-1)}$, as described in the
  following diagram

$$
\xymatrix{
& 0 \ar[d] & 0 \ar[d] & 0 \ar[d]\\
0\ar[r] & \EE_{c+2(n-1)}	\ar[r] \ar[d] & \OO_{\PP^2}^{2(c+n-1)} \ar[d] \ar[r]^{A_n} & \OO_{\PP^n}(1)^c \ar[d] \ar[r] & 0 \\
0\ar[r] & \EE_{c+2(n-1)+i}\ar[r] \ar[d] & \OO_{\PP^n}^{2(c+n-1)+i} \ar[d] \ar[r]^{[A_n|C]} & \OO_{\PP^n}(1)^c \ar[d] \ar[r] & 0 \\
0 \ar[r] & \OO_{\PP^n}^{i} \ar[r] \ar[d] & \OO_{\PP^n}^{i} \ar[r] \ar[d] & 0\\
& 0 & 0
}
$$
where $[A_n|C]$ is constructed from $A_n$ adding $i$ linearly independent columns of linear forms. Again we can add at most $(c-2)(n-1)$ linearly independent columns because, $h^1(\EE)_{c+2(n-1)} = (c-2)(n-1)$. So, when $i$ sweeps the interval $[0,(c-2)(n-1)]$, we get all possible values of $r\in [c+2(n-1),cn]$.
The fact that $\EE_{c+2(n-1)+i}$ is a 1-type uniform Steiner bundle on $\PP^n$ follows from the exact sequence
$$·0 \longrightarrow \EE_{c+2(n-1)} \longrightarrow \EE_{c+2(n-1)+i} \longrightarrow  \OO_{\PP^n}^{i}  \longrightarrow 0$$ and the fact that $\EE_{c+2(n-1)}$ is a 1-type uniform Steiner bundle (see Example 12).
\end{proof}

Let us now recall the geometric construction that allowed Dr\'{e}zet, in \cite{D}, to construct examples of indecomposable uniform vector bundles of rank $2n$ on $\PP^n$ which are not homogeneous. The natural generalization of such a construction will be of extreme importance for our work. Indeed, it will imply that for any $1$-type uniform bundle $\EE$ on $\PP^n$, it is possible to find a short exact sequence
$$
0 \longrightarrow \OO_{\PP^n}^{\alpha} \longrightarrow \EE^\lor \longrightarrow \FF^\lor \longrightarrow 0,
$$
with $\alpha \geq 0$ and $\FF$ a 1-type uniform bundle of minimal rank, with respect to the bound obtained in the previous result, i.e. $\rk \FF = c + 2(n-1)$.

Let $V$ an $n+1$-dimensional $\mathbb{K}$-vector space and $\PP^n = \PP(V)$ the associated projective space of lines. For any integer $p\geq 2$, take the symmetric power $S^pV$ and $H \subset S^pV$ a $\mathbb{K}$-vector subspace, recalling that the fibre of $\OO_{\PP(V)}(-p)$ on the point $x\in \PP(V)$ is the line $x^p \subset S^pV$. Define a morphism of vector bundles
$$
f_H: \OO_{\PP(V)}(-p) \longrightarrow (S^pV/H)\otimes\OO_{\PP(V)}
$$
such that $(f_H)_x(y)=y+H$, with $x\in \PP(V)$ and $y\in (\OO_{\PP(V)}(-p))_x$. We have (see \cite[Lemme 1]{D}) that $\EE(H) := \ker f_H^*$ is a 1-type uniform vector bundle on $\PP^n$ if and only if $S^pD\cap H = \{0\}$ for any plane $D \subset V$. Indeed, such a geometrical condition is equivalent to asking for the surjectivity of the map in cohomology
$$
H^1(\EE(H)(-1)^{n-1}) \longrightarrow H^1(\EE(H))
$$
which was described in the proof of Theorem \ref{mainthm2}.
Moreover, the union of all $S^pD$ of $S^pV$, varying $D$ in the Grasmmannian of planes of $V$, is a homogeneous closed subvariety of $S^pV$ of dimension $2n+p-1$.

As a natural generalization of such construction, we consider $c$ copies of the line bundle $\OO_{\PP(V)}(-p)$ and a map
$$
f_H^p: \OO_{\PP(V)}(-p)^c \longrightarrow ((S^pV)^{\oplus c}/H)\otimes\OO_{\PP(V)},
$$
with $H \subset (S^pV)^{\oplus c}$ a $\mathbb{K}$-vector subspace. As in the previous case, if we denote $\EE(H) := \ker (f_H^p)^*$, we have that $\EE(H)$  is a 1-type uniform bundle if and only if, for any plane $D \subset V$ we have an injective morphism $D^{\oplus c} \rightarrow (S^pV)^{\oplus c}/H$, which is given by the map in cohomology induced by $f_H^p$, when restricted to the line in $\PP^n$ associated to the $\mathbb{K}$-vector space $D$.

Notice that we have already proven, recall Theorem \ref{mainthm2}, that the minimal rank of $1$-type uniform bundle $\EE$ on $\PP^n$, with $c_1(\EE) = -c$, is equal to $c +2(n-1)$. This means that we have already proven that, if we consider $p=1$, we get $\dim H \leq (n+1)c -2c -2n +2$ and moreover, we have shown explicit examples for which the subspace $H$ reaches the highest possible dimension. In this case, we will denote the vector subspace by $H_{max}$.

The following result is a direct consequence of the described geometric interpretation of the 1-type uniformity joint with the existence of examples of minimal rank that we have constructed.

\begin{Theorem}\label{thmGeom}
Let $\EE$ be a 1-type uniform vector bundle on $\PP^n$, with $c_1(\EE) = -c$ and rank $r\geq c +2(n-1) $. Then, the vector bundle $\EE$ always fits into the following short exact sequence
$$
0 \longrightarrow \FF \longrightarrow \EE \longrightarrow \OO_{\PP^n}^\alpha \longrightarrow 0,
$$
with $\alpha \in \mathbb{N}$ and where $\FF$ is a 1-type uniform vector bundle with $\rk(\FF) = c+2(n-1)$, i.e. the minimum allowed.
\end{Theorem}
\begin{Remark}
The previous result tells us that whenever we have a 1-type uniform vector bundle $\EE$ on $\PP^n$, it always possible to eliminate the right number of linear combinations of columns of its defining matrix in order to obtain a minimal rank 1-type uniform bundle $\FF$ on $\PP^n$. This, combined with Theorem \ref{mainthm2}, means that starting from any $1$-type uniform bundle, we can add (respectively delete) columns in order to get a 1-type uniform bundle of maximal (respectively minimal) rank. Therefore we can think of any two 1-type uniform bundles as connected by ``paths'' representing addition or deletion of linear combinations of columns of the defining matrices.
\end{Remark}
\begin{proof}[Proof of Theorem \ref{thmGeom}]
The minimal rank examples introduced in Example \ref{minRankPn} tell us that $$\dim H_{max}=(n+1)c -2c -2n +2.$$ Therefore when we consider another $\mathbb{K}$-vector subspace $H\subset V^{\oplus c}$ with $\dim H < \dim H_{max}$, the geometric interpretation implies that we still have room to extend it to reach $\dim H_{max}$, maintaining the 1-type uniformity.
\end{proof}


\section{Families of uniform bundles which are not homogeneous}
Our goal is to prove now that many of the 1-type uniform bundles described in the previous section are indeed not homogeneous. Throughout this part, it will be important to underline the value of the first Chern class of the considered bundles, therefore we will specify its opposite as an index.\\ Specifically, this entire section is devoted to prove the following result
\begin{Theorem}
Let $\EE_c$ be a 1-type uniform Steiner bundle on $\PP^n$ with $c_1(\EE_c) = -c$, $c\geq 4$ and rank $c+2(n-1)$ (the minimum possible), as defined in Examples \ref{ex-1type} and \ref{minRankPn}. Then $\EE_c$ is not homogeneous.
\end{Theorem}
We will divide the proof in several steps, starting from the case of the projective plane $\PP^2$. The considered bundles will all have their defining matrix as described in Example \ref{ex-1type}.
\begin{Lemma}\label{Claim1}
Let $\EE_c$ be a 1-type uniform bundle on $\PP^2$ with $c_1(\EE_c) = -c$ and rank $c+2$, then 
$$
h^1(\EE_c(t+1))\leq \max(0,h^1(\EE_c(t))-2),
$$
for every $t\geq -1$.
\end{Lemma}
\begin{proof}
Recall that, in the proof of Theorem \ref{mainthm1}, we have that being 1-type uniform is equivalent to having a surjective map of vector bundles in the dual projective plane 
$$
H^1(\EE_c(t)) \otimes \OO_{(\PP^2)^*}(-1) \longrightarrow H^1(\EE_c(t+1)) \otimes \OO_{(\PP^2)^*}\longrightarrow 0
$$
for every $t\geq -1$, which defines a Steiner bundle and therefore implies that $h^1(\EE_c(t)) - h^1(\EE_c(t+1) \geq 2.$
\end{proof}
\begin{Lemma}
Let $\EE_c$ be a 1-type uniform bundle on $\PP^2$ with $c_1(\EE_c) = -c$ and rank $c+2$, then $h^0\left(\EE_c(1)\right) = c+2$.
\end{Lemma}
\begin{proof} 
For small values of $c$, for example up to 8, it is possible to prove the result by direct computation, or with Macaulay2, using the defining matrix of $\EE_c$. We will hence suppose that the result holds for the case $c-1$ and complete the proof by induction.
 
Deleting the last row of the defining matrix of the bundle $\EE_c$, we get the following commutative diagram, thanks to the Snake Lemma,
$$
\xymatrix{
& & & 0\ar[d] \\
& 0 \ar[d] & 0 \ar[d] & \OO_{\PP^2}(1) \ar[d]\\
0 \ar[r] & \EE_c \ar[r] \ar[d] & \OO^{2c+2}_{\PP^2} \ar[d] \ar[r] & \OO_{\PP^2}(1)^c \ar[d]\ar[r] & 0\\ 
0 \ar[r] & \EE_{c-1} \oplus \OO_{\PP^2}^2 \ar[r] \ar[d] & \OO^{2c+2}_{\PP^2} \ar[d] \ar[r] & \OO_{\PP^2}(1)^{c-1} \ar[d]\ar[r] & 0\\
& \OO_{\PP^2}(1) \ar[d] & 0 & 0 \\
& 0
}
$$
Notice that such diagram is constructed deleting the last row of the matrix which defines $\EE_c$, obtaining the matrix that defines $\EE_{c-1}$ plus two columns of zeros, which give us the summand $\OO_{\PP^2}^2$ in the kernel of the middle row. Hence we have that
$$
h^0(\EE_c(1)) = h^0(\EE_{c-1}(1)) + 2 \cdot h^0(\OO_{\PP^2}(1)) - h^0(\OO_{\PP^2}(2)) +  h^1(\EE_c(1)) - h^1(\EE_{c-1}(1))
$$
Due to the previous lemma and the induction hypothesis, we have that $h^1(\EE_c(1)) \leq c-4$ and $h^1(\EE_{c-1}(1)) = c-5$, which implies that $h^0(\EE_c(1)) \leq c+2$.\\
Moreover, looking at the defining matrix of $\EE_c$, it is possible to compute directly $c+2$ linear syzygies, hence $h^0(\EE_c(1)) \geq c+2$; hence, we have an equality. 
\end{proof}
We obtain that the resolution of a 1-type uniform bundle $\EE_c$ on $\PP^2$ with $c_1(\EE_c) = -c$ and rank $c+2$ is of the following type
\begin{equation}\label{resolutionMin}
0 \rightarrow \bigoplus_{j=1}^{s} \OO_{\PP^2}(b_j) \rightarrow \bigoplus_{i=1}^{s} \OO_{\PP^2}(a_i) \oplus \OO_{\PP^2}(-1)^{c+2} \rightarrow \EE_c \rightarrow 0
\end{equation}
with $a_i < -1$ and $b_j<-2$.
\begin{Lemma}
Let $\EE_c$ be a 1-type uniform bundle on $\PP^2$ with $c_1(\EE_c) = -c$ and rank $c+2$, then 
$$
h^1(\EE_c(t+1))= \max(0,h^1(\EE_c(t))-2),
$$
for every $t\geq -1$.
\end{Lemma}
\begin{proof}
Let us compute $h^1(\EE_c(2))$, knowing, by the resolution in \ref{resolutionMin} combined with the defining sequence
$$
0 \rightarrow \EE_c \rightarrow \OO_{\PP^2}^{2c+2} \rightarrow \OO_{\PP^2}(1)^c \rightarrow 0,
$$
that $h^1(\EE_c(1))=c-4$. 

Using again the resolution (\ref{resolutionMin}), we have that $h^0(\EE_c(2)) \geq 3c+6$; moreover, by Lemma \ref{Claim1} and the defining sequence we get that $h^0(\EE_c(2))\leq 3c+6$, giving us the two equalities $h^0(\EE_c(2))= 3c+6$ and $h^1(\EE_c(2))= c-6$. The statement is proven iterating this technique.
\end{proof}
Combining the previous results we obtain the following one.
\begin{Proposition}\label{prop-resE}
Let $\EE_c$ be a 1-type uniform bundle on $\PP^2$ with $c_1(\EE_c) = -c$ and rank $c+2$, then
\begin{itemize}
\item if $c$ is odd, then $$
h^1\left(\EE_c\left(\frac{c-3}{2}\right)\right)=1 \mbox{ and } h^1(\EE_c(t))=0 \mbox{ for } t > \frac{c-3}{2},
$$ 
moreover, its resolution is given by the short exact sequence
$$
0 \rightarrow  \OO_{\PP^2}\left(-\frac{c-3}{2}-3\right) \rightarrow \OO_{\PP^2}\left(-\frac{c-3}{2}-1\right) \oplus \OO_{\PP^2}(-1)^{c+2} \rightarrow \EE_c \rightarrow 0
$$
\item if $c$ is even, then $$
h^1\left(\EE_c\left(\frac{c-4}{2}\right)\right)=2 \mbox{ and } h^1(\EE_c(t))=0 \mbox{ for } t > \frac{c-4}{2},
$$ 
moreover, its resolution is given by the short exact sequence
$$
0 \rightarrow  \OO_{\PP^2}\left(-\frac{c-4}{2}-3\right)^2 \rightarrow \OO_{\PP^2}\left(-\frac{c-4}{2}-2\right)^2 \oplus \OO_{\PP^2}(-1)^{c+2} \rightarrow \EE_c \rightarrow 0
$$
\end{itemize}
\end{Proposition}
\begin{proof}
The cohomology dimensions follow from the previous Lemmas. The resolutions are obtained taking into consideration the rank and the Chern classes of $\EE_c$.
\end{proof}
Let us conclude this section proving that, if $c\geq 4$, our families of uniform bundles are indeed non-homogeneous.\\
Take $c$ odd and suppose that $\EE_c$ is homogeneous, i.e. for any $t \in \PP GL(3,\KK)$ we have 
$t^* \EE_c \simeq \EE_c$. This isomorphism extends to an isomorphism of the resolution defining the two bundles
$$
\xymatrix{
0 \ar[r] & \OO_{\PP^2}\left(-\frac{c-3}{2}-3\right) \ar[d]_{\lambda} \ar[r]^(.38){\alpha} &  \OO_{\PP^2}\left(-\frac{c-3}{2}-1\right) \oplus \OO_{\PP^2}(-1)^{c+2} \ar[r] \ar[d]^{\left[\begin{array}{cc} \mu & * \\ 0 & * \end{array}\right]} & \EE_c \ar[r] \ar[d]^\simeq & 0\\
0 \ar[r] & \OO_{\PP^2}\left(-\frac{c-3}{2}-3\right) \ar[r]^(.38){t(\alpha)} &  \OO_{\PP^2}\left(-\frac{c-3}{2}-1\right) \oplus \OO_{\PP^2}(-1)^{c+2} \ar[r] & t^* \EE_c \ar[r] & 0
}
$$ 
where we denote by $t(\alpha)$ the change of coordinates applied to the matrix defining the bundle map, and $\lambda, \mu \in \KK$. Denoting by $q$ the degree 2 homogeneous form in the first entry of the matrix representing $\alpha$ and by $t(q)$ the corresponding form after the change of coordinates, we would have that for any change of coordinates $t$, we would be able to find $\lambda$ and $\mu$ in the field $\KK$ such that $t(q) = \mu \cdot q \cdot \lambda$, which is clearly impossible. The even case is analogous and this proves the result for the case of the projective plane $\PP^2$.
\begin{Remark}
Let us discuss the case $c=3$, which we have excluded because it indeed defines a homogeneous bundle.\\
By Proposition \ref{prop-resE}, we have that a resolution of $\EE_3$ is given by
$$
0 \rightarrow \OO_{\PP^2}(-3) \stackrel{M}{\rightarrow} \OO_{\PP^2}(-1)^6 \rightarrow \EE_3 \rightarrow 0
$$
which implies that we have
$$
M^T = \left[q_1 \: q_2 \: q_3 \: q_4 \: q_5 \: q_6 \right] 
$$
where the $q_i$'s denotes a basis for the vector space of homogeneous polinomials of degree two in three variables. For any change of coordinated $t$, it is then possible to recover $t^*(M)$ with linear combinations of the rows of $M$, i.e. we have a commutative diagram 
$$
\xymatrix{
0 \ar[r] & \OO_{\PP^2}(-3) \ar[d]_{\lambda} \ar[r]^(.5){\alpha} &  \OO_{\PP^2}(-1)^{6} \ar[r] \ar[d]^\mu & \EE_3 \ar[r] \ar[d]^\simeq & 0\\
0 \ar[r] & \OO_{\PP^2}(-3) \ar[r]^(.5){t(\alpha)} &   \OO_{\PP^2}(-1)^{6} \ar[r] & t^* \EE_3 \ar[r] & 0
}
$$ 
with $\mu \in GL(6,\KK)$ representing such linear combination of rows, and therefore $\EE_3$ is homogeneous.
\end{Remark}
Finally, let us observe that if we suppose that $\EE_c$ is a homogeneous, 1-type uniform vector bundle on $\PP^n$, with $c_1(\EE_c)=-c$ and rank $c+2(n-1)$, as defined in Example \ref{minRankPn}, its restriction on the projective plane $\PP^2$ given by the equations $\{x_3=\cdots=x_n=0\}$ gives us
$$
(\EE_c)_{|\PP^2} \simeq \tilde{\EE}_c \oplus \OO_{\PP^2}^{2n-4},
$$
where $\tilde{\EE}_c$ a homogeneous, 1-type uniform vector bundle on $\PP^2$, with $c_1(\tilde{\EE}_c)=-c$ and rank $c+2$, as defined in Example \ref{ex-1type}. We have previosuly shown that the bundle $\tilde{\EE}_c$ are not homogeneous for $c\geq 4$, and this concludes the proof of the result stated at the beginning of this section.

\section{The $k$-type case}
In this section, we consider the higher $k$-type uniform cases.
Let us start noticing that the technique used to construct $1$-type uniform Steiner bundles on $\PP^2$  generalizes and allows us to build $k$-type uniform Steiner bundles on $\PP^2$. This will lead to a very interesting remark on the moduli space of such bundles.
\begin{Example}\label{minkRankPn}
Let us consider the direct sum of symmetric powers, up to $k$, of the cotangent bundle on $\PP^2$ tensored by $\OO_{\PP^2}(1)$, i.e.
$$
\EE \simeq \bigoplus_{j=1}^k \left(S^j\left(\Omega_{\PP^2}(1)\right)\right)^{\alpha_j}
$$
where we have chosen $j\geq 1$ to avoid trivial summands. Obviously, $\EE$ is a $k$-type uniform Steiner  bundle on $\PP^2$. Looking at the defining matrix $B_j$ of $S^j\left(\Omega_{\PP^2}(1)\right)$ as a Steiner bundle, we observe that it can be represented to have always three columns, each one containing only, respectively, the variable $x, y$ and $t$. Moreover, the three variables are also on different rows. For instance, in
$$
B_3=
 \left[
\begin{array}{cccccccccccccc}
x & y & t & 0 & 0 & 0 & 0 & 0 & 0 & 0\\
0 & x & 0 & y & t & 0 & 0 & 0 & 0 & 0\\
0 & 0 & x & 0 & y & t & 0 & 0 & 0 & 0\\
0 & 0 & 0 & x & 0 & 0 & y & t & 0 & 0\\
0 & 0 & 0 & 0 & x & 0 & 0 & y & t & 0\\
0 & 0 & 0 & 0 & 0 & x & 0 & 0 & y & t
\end{array}
\right]
$$
in the 1st column (resp. 7th column, 10th column) only the variable $x$ appears (resp. $y$, $t$) and they appear in rows 1, 4 and 6, respectively.

Therefore, it is possible to generalize the technique, presented in Section \ref{sec-1type}, in order to produce new examples of $k$-type uniform Steiner bundles on $\PP^2$. If $\sum_{j=1}^k \alpha_j := \alpha \geq 3$, then the defining matrix of $\EE$  consists of, at least three, diagonal blocks of type $B_i$, and we have exactly $\alpha_j$ blocks of type $B_j$ for any $j=1,\ldots,k$.
Now, starting from the third block, henceforth we can do it for $\alpha - 2$ blocks, we can erase the column which contains only one variable,  for example the variable $t$, and add this variable, on the same line it was on the deleted column, to the two previous blocks. In the first one, add it in the column related to one of the remaining variables, let us say $x$, and in the second block we will add it in the column related to the remaining variable, let us say $y$. Arguing as in Example \ref{ex-1type} we prove that the constructed bundles are indeed $k$-type uniform Steiner bundles on $\PP^2$ which are not homogeneous.

\vskip 2mm
Let us look at the following specific example which will clarify the process we have just explained. We consider the vector bundle $$\EE \simeq S^3\left(\Omega_{\PP^2}(1)\right) \oplus \left(S^2\left(\Omega_{\PP^2}(1)\right)\right)^2 \oplus \Omega_{\PP^2}(1)$$
on $\PP^2$ defined by a short exact sequence of type
$$
0 \longrightarrow \EE \longrightarrow \OO_{\PP^2}^{25} \stackrel{A}{\longrightarrow} \OO_{\PP^2}(1) ^{13} \longrightarrow 0,
$$
where we represent by blocks the defining matrix:
$$
A =
\left[
\begin{array}{c|c|c|c}
B_3 & 0 & 0 & 0 \\
\hline 0 & B_2 & 0 & 0 \\
\hline 0 & 0 & B_2 & 0 \\
\hline 0 & 0 & 0 & B_1
\end{array}
\right].
$$
Notice that $\EE$ is a rank 12, 3-type uniform Steiner bundle on $\PP^2$ with $c_1(\EE )=-13$ and splitting type  $(-3,-2^3,-1^4,0^4)$ on any line $L\subset \PP^2$.
We choose and delete two columns containing only the variable $t$ and we follow the algorithm  explained before to add the variable $t$ in the suitable entries of the matrix. For instance, we  can choose the last column of $B_1$ and the last column of the second block $B_2$. So, in the original matrix $A$ we delete the columns 22 and 25. Now, we add the variable $t$ in the row 12 (resp. 11) and in the columns 11 and 20 (resp. 1 and 14) and we get  the matrix
$$
A'=
\left[
\begin{array}{ccccccccccccccccccccccccccccccccccccccccccccccccccccccccccccccccccccccccccccccccc}
x & y & t & 0 & 0 & 0 & 0 & 0 & 0 & 0 & 0 & 0 & 0 & 0 & 0 & 0 & 0 & 0 & 0 & 0 & 0 & 0 & 0 \\
0 & x & 0 & y & t & 0 & 0 & 0 & 0 & 0 & 0 & 0 & 0 & 0 & 0 & 0 & 0 & 0 & 0 & 0 & 0 & 0 & 0 \\
0 & 0 & x & 0 & y & t & 0 & 0 & 0 & 0 & 0 & 0 & 0 & 0 & 0 & 0 & 0 & 0 & 0 & 0 & 0 & 0 & 0 \\
0 & 0 & 0 & x & 0 & 0 & y & t & 0 & 0 & 0 & 0 & 0 & 0 & 0 & 0 & 0 & 0 & 0 & 0 & 0 & 0 & 0 \\
0 & 0 & 0 & 0 & x & 0 & 0 & y & t & 0 & 0 & 0 & 0 & 0 & 0 & 0 & 0 & 0 & 0 & 0 & 0 & 0 & 0 \\
0 & 0 & 0 & 0 & 0 & x & 0 & 0 & y & t & 0 & 0 & 0 & 0 & 0 & 0 & 0 & 0 & 0 & 0 & 0 & 0 & 0 \\
0 & 0 & 0 & 0 & 0 & 0 & 0 & 0 & 0 & 0 & x & y & t & 0 & 0 & 0 & 0 & 0 & 0 & 0 & 0 & 0 & 0 \\
0 & 0 & 0 & 0 & 0 & 0 & 0 & 0 & 0 & 0 & 0 & x & 0 & y & t & 0 & 0 & 0 & 0 & 0 & 0 & 0 & 0 \\
0 & 0 & 0 & 0 & 0 & 0 & 0 & 0 & 0 & 0 & 0 & 0 & x & 0 & y & t & 0 & 0 & 0 & 0 & 0 & 0 & 0 \\
0 & 0 & 0 & 0 & 0 & 0 & 0 & 0 & 0 & 0 & 0 & 0 & 0 & 0 & 0 & 0 & x & y & t & 0 & 0 & 0 & 0 \\
0 & 0 & 0 & 0 & 0 & 0 & 0 & 0 & 0 & 0 & 0 & 0 & 0 & 0 & 0 & 0 & 0 & x & 0 & y & t & 0 & 0 \\
t & 0 & 0 & 0 & 0 & 0 & 0 & 0 & 0 & 0 & 0 & 0 & 0 & t & 0 & 0 & 0 & 0 & x & 0 & y & 0 & 0 \\
0 & 0 & 0 & 0 & 0 & 0 & 0 & 0 & 0 & 0 & t & 0 & 0 & 0 & 0 & 0 & 0 & 0 & 0 & t & 0 & x & y
\end{array}
\right]
$$
The Steiner bundle $\EE'$ defined as the kernel of the map represented by $A'$ is a 3-type uniform bundle of rank 10 and $c_1(\EE')=-13$. Indeed, it is possible to compute by hand that considering for example each possible line in three different steps $t =\alpha x + \beta y$, $x = \alpha y$ and $y=0$, with $\alpha, \beta \in \KK$, the splitting is of type $(-3,-2,-2,-2,-1,-1,-1,-1,0,0)$.
\end{Example}

As a direct consequence of the families introduced in the Examples \ref{minRankPn} and \ref{minkRankPn} we have the following result.

\begin{Proposition}
There exist irreducible families of Steiner bundles $\EE$ on $\PP^n$ with fixed rank $r$ and first Chern class $c_1(\EE)=-c$ which contain uniform Steiner bundles of different $k$-type.
\end{Proposition}
Moreover, as noticed in Section \ref{sec-1type}, it is possible to connect $k$-type uniform bundles (in this case varying the $k$ in the steps) by paths that represent the addition or deletion of linear combination of columns of the defining matrices.

Our next goal is to establish upper and lower bounds for the rank of a $k$-type uniform Steiner bundle $\EE$ on $\PP^2$ with fixed first Chern class.  Differently from the case $k=1$, we will study separately upper and lower bounds. First, we will discuss a lower bound for the rank.
\begin{Theorem}
Let $\EE$ be a rank $r$, $k$-type uniform Steiner  bundle on $\PP^2$ with $c_1(\EE)=-c$, then $$ r \geq \frac{c+2}{k}. $$
\end{Theorem}
\begin{proof}
As a direct generalization of the case $k=1$, we have that $\EE$ is a $k$-type uniform bundle if and only if the cohomology map
$$
H^1(\EE(k-2)) \stackrel{f_L}{\longrightarrow} H^1(\EE(k-1)) \longrightarrow H^1(\EE_L(k-1)) =0
$$
is surjective for any choice of a line $L \subset \PP^2$. This means that we are looking for the conditions to ensure the existence a surjective map of vector bundles
\begin{equation}\label{steinDualk}
H^1(\EE(k-2)) \otimes \OO_{(\PP^2)^*}(-1) \longrightarrow H^1(\EE(k-1)) \otimes \OO_{(\PP^2)^*}\longrightarrow 0
\end{equation}
defined on the dual projective space $(\PP^2)^*$. As before, see \cite{DK}, this occurs if and only if $h^1(\EE(k-2)) - h^1(\EE(k-1)) \geq 2$. Using the defining short exact sequence of $\EE$ to compute the cohomology groups, observing that we have  that
$$
-c(k+1) + (r+c)k \geq 2 + h^0(\EE(k-1)) - h^0(\EE(k-2)),
$$
and, observing that $h^0(\EE(k-2)) \leq h^0(\EE(k-1))$, we obtain the required inequality, i.e.
$r \geq \frac{c+2}{k}.$
\end{proof}
\begin{Remark}
Unfortunately, the previous bound is not sharp. Let us consider for example the case $c=4$ and $k=2$. Indeed, from the previous bound, it would be possible to obtain a 2-type uniform vector bundle of rank 3. Nevertheless, it is known that all uniform rank 3 vector bundles on $\PP^2$ are of the form (see \cite{E})
$$
\OO_{\PP^2}(a) \oplus \OO_{\PP^2}(b) \oplus \OO_{\PP^2}(c) \:\:\mbox{or}\:\: \Omega_{\PP^2}(a) \oplus \OO_{\PP^2}(b) \:\:\mbox{or}\:\: S^2\Omega_{\PP^2}(a),
$$
each one giving, under our hypothesis, a contradiction.
\end{Remark}
Let us conclude this part focusing on a higher bound for the rank of $k$-type uniform bundles on  $\PP^2$.
\begin{Theorem}
Let $\EE$ be a rank $r$, $k$-type uniform Steiner bundle on $\PP^2$ with $c_1(\EE)=-c$. Then $$r \leq 2c - k^2 +k.$$
\end{Theorem}

\begin{proof}
First of all, let us denote the splitting of $\EE$ in each line by 
$$(-k^{\alpha_k},-k+1^{\alpha_{k-1}},\dots,-1^{\alpha_1},0^{\alpha_0})$$
and notice that being $\EE$ uniform with $\alpha_k>0$, we cannot have gaps in the splitting type of the bundle. (see \cite[Corollary 3, pg 206]{OSS}). Hence, the uniform vector bundle with the maximal rank will be given by a splitting type of the following form
$$
(-k,-k+1,\dots,-2,-1^{\alpha_1},0^{\alpha_0}).
$$
In order to obtain the required Chern class for $\EE$, we have that

$$
\alpha_1 = c - \frac{k(k+1)}{2} +1
$$
and also
$$
r = c - \frac{k(k+1)}{2} +1 + \alpha_0 + k -1.
$$

The integer $\alpha_0$ describes the number of columns of the defining matrix of $\EE$ which depend, after a possible action of the involved groups in the matrix, on a unique linearly independent linear form; let us suppose it to be $t$. Restrict to a line different from the one defined by the vanishing of the previous linear form, suppose it to be $L:=\{y=0\}$ for example.
Observe that the columns, in the defining matrix of $\EE$, where appears the variable $y$ but not the variable $x$, are exactly $\alpha_1 + k -1$ because, restricting at $t=0$, we have a $y$ for each non trivial summand of the bundle restricted on the line. If we now restrict at the line $y=0$, then we must have $\alpha_1 -k-1$ columns that depend only by $t$ and this implies that
$$
\alpha_0\leq \alpha_1 + k -1.
$$
Putting all together, we get $r \leq 2c - k^2 +k$, as required.
\end{proof}
\begin{Remark}
The obtained bound is also not sharp. Indeed, our belief is that we should have $r \leq 2c-k^2+1$, with the equality achieved if and only if $\EE$ is of the form $S^k(\Omega_{\PP^2}(1)) \oplus \Omega_{\PP^2}(1)^{c-\binom{k+1}{2}+1}$. This is actually one of the reason that pushed us to state the conjecture that closes this work.
\end{Remark}

\section{Final remarks and open questions}
In this last section we will discuss some interesting problems that arise from what we have discussed in this work.
A careful reader has already noticed that the family of vector bundles presented in Example \ref{minRankPn} fills all the cases allowed by the rank bounds proven for 1-type uniform bundles. On the other hand, the family of vector bundles introduced in Example \ref{minkRankPn} leaves some ``gaps'' in set of allowed ranks.
This can be justified observing that we are considering families of vector bundles which are all Steiner. In any case, our strong belief is that any uniform vector bundle which is also Steiner can be obtained through our construction. Such statement leads to the following conjecture
\begin{Conjecture}
Any $k$-type uniform vector bundle $\EE$ on $\PP^n$, for any $k$ positive integer, can be obtained as the kernel of the following short exact sequence
$$
0 \longrightarrow \EE \longrightarrow \bigoplus_{j=0}^k S^j\left(\Omega_{\PP^2}(1)\right)^{\alpha_j} \longrightarrow \OO_{\PP^2}^\beta \longrightarrow 0,
$$
with $\alpha_j \geq 0$, for $j=1,\ldots,k-1$, $\alpha_k>0$ and $\beta\geq 0$. Moreover, all such bundles are obtained using the process described in the Examples \ref{minRankPn} and \ref{minkRankPn}.
\end{Conjecture}

\end{document}